\documentclass{article}

\usepackage{amssymb,amsthm}
\usepackage{amsmath}
\usepackage{latexsym}
\usepackage{amssymb}
\usepackage[dvips]{graphics}
\usepackage[dvips]{graphicx}
\newtheorem{theorem}{Theorem}[section]

\newtheorem{lemma}[theorem]{Lemma}
\newtheorem{corollary}[theorem]{Corollary}
\newtheorem{definition}[theorem]{Definition}
\newtheorem{remark}[theorem]{Remark}
\newtheorem{example}[theorem]{Example}

\newcommand{\R}{{\mathbf R}}
\newcommand{\Z}{{\mathbf Z}}

\renewcommand{\int}{\rm Int}
\newcommand{\E}{\mathbb E}

\newcommand{\PP}{{\mathbb P}}
\newcommand{\p}{{\mathbf p}}
\newcommand{\aalpha}{{\mathbf \alpha}}
\newcommand{\Mu}{{\mathcal M}}

\begin{document}

\title{Random simplicial complexes}         
\author{A. Costa and M. Farber}        
\date{February 28, 2015}          
\maketitle

\section{Introduction}
The study of random topological objects (such as random simplicial complexes and random manifolds) is motivated by potential applications to modelling of large complex systems 
in various engineering and computer science applications. 
Random topological objects are also of interest from pure mathematical point view since they can be used for constructing curious examples of topological objects with 
rare combinations of topological properties. 

Several models of random manifolds and random simplicial complexes were suggested and studied recently, see \cite{Ksurvey} for a survey. 
One may mention random surfaces \cite{PS}, random 3-manifodls
\cite{DT}, random configuration spaces of linkages 
 \cite{F}. 
The present paper is was inspired by
the model of random simplicial complexes developed by Linial, Meshulam and Wallach \cite{LM} , \cite{MW}. 
In the first paper \cite{LM} the authors 
studied an analogue of the classical Erd\"os - R\'enyi \cite{ER} model of random graphs in the situation of 2-dimensional simplicial complexes. In the following paper 
\cite{MW} 
a more general model of $d$-dimensional 
random simplicial complexes was studied. The random simplicial complexes of \cite{LM} and \cite{MW} have the complete $(d-1)$-skeleton and their \lq\lq randomness is concentrated in the top dimension\rq\rq.
More specifically, one starts with the full $(d-1)$-skeleton of an $(n-1)$-dimensional simplex and adds $d$-faces at random, independently of each other, with probability $p$. 

A different model of random simplicial complexes was studied by M. Kahle \cite{Kahle1}, \cite{Kahle3} and by some other authors. 
These are clique complexes of random Erd\"os - R\'enyi graphs; here one takes a random graph in the 
Erd\"os - R\'enyi model and declares as a simplex every subset of vertices which form a {\it clique}, i.e. such that every two vertices of the subset are connected by an edge. 
Compared with the Linial - Meshulam model, the clique complex has \lq\lq randomness\rq\rq\,  in dimension one but it influences the structure in all the higher dimensions. 

In this paper we propose a more general and more flexible model of random simplicial complexes with randomness in all dimensions. 
We start with a set of $n$ vertices and retain each of them with probability $p_0$; on the next step we connect every pair of retained vertices by an edge with probability $p_1$, and then fill in every triangle in the obtained random graph with probability $p_2$, and so on. 
As the result we obtain a random simplicial complex depending on the set of probability parameters 
$$(p_0, p_1, \dots, p_r), \quad 0\le p_i\le 1.$$
Our multi-parameter random simplicial complex includes both Linial-Meshulam and random clique complexes as special case. 
Topological and geometric properties of this random simplicial complex depend on the whole set of parameters and their thresholds can be understood as convex subsets and not as single numbers as in all the previously studied models. We mainly focus on foundations and on containment properties of our multi-parameter random simplicial complexes. One may associate to any finite simplicial complex $S$ a reduced density domain $\tilde \mu(S)\subset \R^r$ (which is a convex domain) which fully controls information about the values of the multi-parameter for which the random complex 
contains $S$ as a simplicial subcomplex. 
We also analyse balanced simplicial complexes and give positive and negative examples. We apply these results to describe dimension of a random simplicial complex. 

In a following paper we shall address other topological and geometric properties of random simplicial complexes depending on multiple parameters (such as their homology and the fundamental group).

The authors thank Thomas Kappeler for useful discussions. 

\section{The definition and basic properties.}  

\subsection{The model} \label{1.1}
   Let $\Delta_n$ denote the simplex with the vertex set $\{1, 2, \dots, n\}$. We view $\Delta_n$ as an abstract simplicial complex of dimension $n-1$. 
Given a simplicial subcomplex $Y\subset \Delta_n$, we denote by $f_i(Y)$ the number of {\it $i$-faces} of $Y$ (i.e. $i$-dimensional simplexes of $\Delta_n$ contained in $Y$). 
\begin{definition}
An external face of a subcomplex $Y\subset \Delta_n$ is a simplex $\sigma \subset \Delta_n$ such that $\sigma \not\subset Y$ but the boundary of $\sigma$ is contained in $Y$, 
$\partial \sigma \subset Y$. 
\end{definition}
We shall denote by $e_i(Y)$ the number of 
$i$-dimensional external faces of $Y$. 
Note that for $i=0$, 
we have $e_0(Y)+f_0(Y)=n$ and for $i>0$,
$$f_i(Y)+e_i(Y)\le \binom n {i+1}.$$

Fix an integer $r>0$ and a sequence $${\mathbf p}=(p_0, p_1, \dots, p_r)$$ of real numbers satisfying $$0\le  p_i\le 1.$$ Denote  
$$q_i=1-p_i.$$
We consider the probability space $Y_r(n; {\mathbf p})$ consisting of all subcomplexes $$Y\subset \Delta_n^{(r)}$$ with $\dim Y\le r$ where the probability function
$$\PP_r: Y_r(n;{\mathbf p})\to \R$$ is given by the formula
\begin{eqnarray}\label{def1}
\PP_r(Y) = \prod_{i=0}^r p_i^{f_i(Y)}\cdot \prod_{i=0}^r q_i^{e_i(Y)}
\end{eqnarray}
for $Y\in Y_r(n; {\mathbf p}).$
We shall show below that $\PP_r$ is indeed a probability function, i.e. 
\begin{eqnarray}\label{sum1}
\sum_{Y\subset \Delta_n^{(r)}}\PP_r(Y) = 1, 
\end{eqnarray}
see Corollary \ref{prob}. 

If $p_i=0$ for some $i$ then according to (\ref{def1}) we shall have $\PP_r(Y)=0$ unless $f_i(Y)=0$, i.e. if $Y$ contains no simplexes of dimension $i$ (in this case $Y$ contains no simplexes
of dimension $\ge i$). Thus, if $p_i=0$ the probability measure $\PP_r$ is concentrated on the set of subcomplexes of $\Delta_n$ of dimension $<i$. 

In the special case when one of the probability parameters satisfies $p_i=1$ one has $q_i=0$ and from formula (\ref{def1}) we see $\PP_r(Y)=0$ unless $e_i(Y)=0$, i.e. if the subcomplex $Y\subset \Delta_n^{(r)}$ has no external faces of dimension $i$. In other words, we may say that if $p_i=1$ the measure $\PP_r$ is concentrated on the set of complexes satisfying $e_i(Y)=0$, i.e. such that any boundary of the $i$-simplex in $Y$ is filled by an $i$-simplex. 

\begin{lemma}\label{cont}
Let $$A\subset B\subset \Delta_n^{(r)}$$ be two subcomplexes satisfying the following condition: the boundary of any external face of $B$ of dimension $\le r$ 
is contained in $A$. Then 
\begin{eqnarray}\label{twosided}
\PP_r(A\subset Y\subset B) = \prod_{i=0}^r p_i^{f_i(A)} \cdot \prod_{i=0}^r q_i^{e_i(B)}.
\end{eqnarray}
\end{lemma}

\begin{proof} We act by induction on $r$. 
For $r=0$, $A\subset B$ are discrete sets of vertices and the condition of the Lemma is automatically satisfied (since the boundary of any 0-face is the empty set). A subcomplex $Y\subset \Delta_n^{(0)}$ satisfying $A\subset Y\subset B$ is determined by a choice of $f_0(Y)-f_0(A)$ vertices out of $f_0(B)-f_0(A)$ vertices. Hence using formula (\ref{def1}), 
\begin{eqnarray*}
\PP_0(A\subset Y\subset B) &=& \sum_{k=0}^{f_0(B)-f_0(A)} \binom {f_0(B)-f_0(A)} k \cdot p_0^{f_0(A)+k}q_0^{n-f_0(A)-k}\\
&=& p_0^{f_0(A)} \cdot q_0^{n-f_0(A)}\cdot \left(1+ \frac{p_0}{q_0}\right)^{f_0(B)-f_0(A)} \\
&=& p_0^{f_0(A)} \cdot q_0^{n-f_0(B)}\\
&=& p_0^{f_0(A)} \cdot q_0^{e_0(B)},
\end{eqnarray*}
as claimed. 

Now suppose that formula (\ref{twosided}) holds for $r-1$ and consider the case of $r$. Note the formula 
\begin{eqnarray}\label{ind}
\PP_r(Y) = \PP_{r-1}(Y^{r-1}) \cdot q_r^{g_r(Y)}\cdot \left(\frac{p_r}{q_r}\right)^{f_r(Y)}
\end{eqnarray}
where $g_r(Y)=e_r(Y)+f_r(Y)$ is the number of boundaries of $r$-simplexes contained in $Y$. 
Note that the first two factors in (\ref{ind}) depend only on the skeleton $Y^{r-1}$. 

We denote by $g_r^B(Y)$ the number of $r$-simplexes of $B$ such that their boundary 
$\partial \Delta^r$ lies in $Y$. Clearly the number $g_r^B(Y)$ depends only on the skeleton $Y^{r-1}$. Our assumption that the boundary of any 
external $i$-face of $B$ is contained in $A$ for $i\le r$ implies that for any subcomplex $A\subset Y\subset B$ one has
\begin{eqnarray}\label{indep}
g_r(Y)-g_r^B(Y) = e_r(B).
\end{eqnarray}

A complex $Y$ is uniquely determined by its skeleton $Y^{r-1}$ and by the set of its $r$-faces. Note that, given the skeleton $Y^{r-1}$, the number $f_r(Y)$ is arbitrary satisfying 
$$f_r(A)\subset f_r(Y)\subset g_r^B(Y).$$
Thus using (\ref{ind}) we find that the probability  
\begin{eqnarray*}
\PP_r(A\subset Y\subset B) = \sum_{A\subset Y\subset B} \PP_r(Y) 
\end{eqnarray*}
equals
\begin{eqnarray*}
&\sum\large_{Y^{r-1}}& 
\PP_{r-1}(Y^{r-1})  \cdot q_r^{g_r(Y)}\cdot 
\sum_{k=0}^{g_r^B(Y)-f_r(A)} 
\binom {g_r^B(Y)-f_r(A)} k \cdot \left(\frac{p_r}{q_r}\right)^{f_r(A)+k} \\
&=& \sum_{Y^{r-1}} 
\PP_{r-1}(Y^{r-1})  \cdot q_r^{g_r(Y)}\cdot \left(\frac{p_r}{q_r}\right)^{f_r(A)}\cdot
\left(1+\frac{p_r}{q_r}\right)^{g_r^B(Y)-f_r(A)}\\
&=& \sum_{Y^{r-1}} 
\PP_{r-1}(Y^{r-1})  \cdot p_r^{f_r(A)} \cdot q_r^{g_r(Y)-g_r^B(Y)}\\
&=& p_r^{f_r(A)}\cdot q_r^{e_r(B))}\cdot \sum_{Y^{r-1}} \PP_{r-1}(Y^{r-1}).
\end{eqnarray*}
Here we used the equation (\ref{indep}). 
Next we may combine the obtained equality with the inductive hypothesis 
$$\PP_{r-1}({A^{r-1}\subset Y^{r-1}\subset B^{r-1}}) = \prod_{i=0}^{r-1} p_i^{f_i(A)}\cdot \prod_{i=0}^{r-1}q_i^{e_u(B)}$$ 
to obtain
(\ref{twosided}). 
\end{proof}

Note that the assumption that any external face of $B$ is an external face of $A$ is essential in Lemma \ref{cont}; the lemma is false without this assumption. 

Taking the special case $A=\emptyset$, $B=\Delta_n^{(r)}$ in (\ref{twosided}) we obtain the following Corollary confirming the fact that $\PP_r$ is a probability function.
\begin{corollary}\label{prob}
$$\sum_{Y\subset \Delta_n^{(r)} }\PP_r(Y) = 1. $$
\end{corollary}

\begin{example}\label{emptyset}
{\rm
The probability of the empty subcomplex $Y=\emptyset$ equals
$$\PP(Y=\emptyset) = (1-p_0)^n.$$

If $p_0\to 0$ then 
$\PP(Y=\emptyset) = (1-p_0)^n \sim e^{-p_0n}.$
Hence, if $np_0\to 0$ then $\PP(Y=\emptyset)\to 1$, i.e. we may say that $Y=\emptyset$, a.a.s.

If $p_0=c/n$ then 
$$\PP(Y=\emptyset) = (1-c/n)^n \to e^{-c}$$ as $n\to \infty$. 
Thus, we see that for $p_0=c/n$ the empty subset appears with probability $e^{-c}$, a.s.s.  

Since we intend to study non-empty large random simplicial complexes, we shall always assume that $p_0=\frac{\omega}{n}$ where $\omega$ tends to $\infty$. 
}
\end{example}

\subsection{The number of vertices of $Y$}

\begin{lemma} Consider a random simplicial complex $Y\in Y_r(n,\p)$, where $\p=(p_0, p_1, \dots, p_r)$ is the probability multiparameter.  
Assume that $p_0=\omega/n$ where $\omega \to \infty$. Then the number of vertices $f_0(Y)$ of $Y$ is approximately $\omega$; 
more precisely for any $0<\epsilon<1/2$ a.a.s. one has
\begin{eqnarray}
(1-\delta)\omega \le f_0(Y) \le (1+\delta)\omega,
\end{eqnarray}
where $\delta=\omega^{-1/2 +\epsilon}$.
\end{lemma}
\begin{proof} 
For a vertex $i\in \{1, \dots, n\}$, denote by $X_i: Y_r(n, \p)\to \R$ the random variable such that $X_i(Y) =1$ if $i\in Y$ and $X_i(Y)=0$ if $i\notin Y$. Then $f_0=\sum_i X_i$ and 
by Lemma \ref{cont}, $\E(X_i)=p_0$. Hence, 
$\E(f_0)=np_0=\omega$. 

The variance of $f_0$ equals $V(f_0)= \sum_{i,j}\E(X_iX_j) - \E(f_0)^2.$ By Lemma \ref{cont}, $\E(X_iX_j)= p_0^2$ for $i\not=j$ and $\E(X_iX_j)= p_0$ for $i=j$. 
Thus, the variance of $f_0$ equals
$$n(n-1)p_0^2 +np_0 -n^2p_0^2 = np_0(1-p_0)=\omega (1-p_0).$$
Applying the Chebychev inequality 
$$\PP\{|f_0-\E(f_0)|\ge \epsilon\}\le V(f_0)/\alpha^2$$
with $\alpha = \delta\omega$ 
we obtain 
$$\PP\{(1-\delta)\omega \le f_0 \le (1+\delta)\omega \} \ge 1- \frac{(1-p_0)}{\delta^2\cdot \omega} \ge 1 - \frac{1}{\omega^{2\epsilon}} \, \to 1.$$
\end{proof}

\subsection{Special cases}

The models of random simplicial complexes which were studied previously contained randomness in a single dimension while our model allows various probabilistic regimes in different dimensions simultaneously. Thus we obtain more flexible constructions of random simplicial complexes. 

The model we consider turns into some well known models in special cases:

When $r=1$ and $\p=(1, p)$ we obtain the classical model of random graphs of Erd\"os and R\'enyi \cite{ER}. 

When $r=2$ and $\p= (1,1,p)$ we obtain the Linial - Meshulam model of random 2-complexes \cite{LM}. 

When $r$ is arbitrary and fixed and $\p = (1, 1, \dots, 1, p)$ we obtain the random simplicial complexes of Meshulam and Wallach \cite{MW}. 

For $r=n-1$ and $\p=(1,p, 1,1, \dots, 1)$ one obtains the clique complexes of random graphs studied in \cite{Kahle1}.

\subsection{Gibbs formalism} 

In this subsection we briefly describe a more general class of models of random simplicial complexes which includes the model of \S \ref{1.1} as a special case. 

On the set of all subcomplexes $Y\subset \Delta_n^{(r)}$, one considers an energy function $H= H_{\beta,\gamma}$  having the form
\begin{eqnarray}
H(Y)= H_{\beta, \gamma} (Y) \ = \, \sum_{i=0}^r \left[ \beta_i f_i(Y) + \gamma_i e_i(Y)\right],
\end{eqnarray}
where $\beta_i$ and $\gamma_i$ are real parameters, $i=0, 1, \dots, r$. The partition function 
\begin{eqnarray}Z=Z_{\beta, \gamma}= \sum_{Y\subset \Delta_n^{(r)}} e^{H_{\beta, \gamma}(Y)}\end{eqnarray}
is a function of the parameters $\beta_i, \gamma_i$ and $n$ and
\begin{eqnarray}\PP_{\beta, \gamma}(Y) \, = \, \frac{1}{Z_{\beta, \gamma}} \, e^{H_{\beta, \gamma}(Y)}\end{eqnarray}
is a probability measure on the set of all subcomplexes $Y\subset \Delta_n^{(r)}$. 
Here the case $r=\infty$ is not excluded; then $Y\subset \Delta_n$ runs over all subcomplexes.

In the special case when the parameters $\beta_i$, $\gamma_i$ satisfy 
 \begin{eqnarray}\label{assump}
e^{\beta_i} + e^{\gamma_i} = 1, \quad i=0, \dots, r.
\end{eqnarray}
we may define the probability parameters $p_i, q_i$ by
\begin{eqnarray}
p_i=e^{\beta_i}, \quad q_i=e^{\gamma_i}.
\end{eqnarray} 
One can easily check that under the assumptions (\ref{assump})  the probability measure 
$\PP_{\beta, \gamma}$ coincides with the measure $\PP_r$ given by (\ref{def1}). 
The relation (\ref{assump}) implies that the partition function $Z_{\beta, \gamma}=1$ equals one, according to Corollary \ref{prob}. 
%

\section{The containment problem}

Consider a random simplicial complex $Y\in Y_r(n,\p)$ where $\p=(p_0, p_1, \dots, p_r)$. As in the classical containment problem for random graphs we ask under which conditions 
$Y$ has a simplicial subcomplex isomorphic to a given $r$-dimen\-sional finite simplicial complex $S$. 
The answer is slightly different from the random graph theory: we associate with $S$ a convex set $$\tilde \Mu(S)\subset \R^{r+1}$$ in the space of exponents of probability parameters 
which (as we show here) is fully responsible for the containment. 

Write
$$p_i=n^{-\alpha_i}, \quad \mbox{where}\quad \alpha_i\ge 0, \quad i=0,1, \dots, r.$$
For simplicity we shall assume in this paper that the numbers $\alpha_i$ are constant (do not depend on $n$). 
We shall use the following notation 
\begin{eqnarray}
\p = n^{-\alpha}, \quad \mbox{where} \quad \alpha=(\alpha_0, \dots, \alpha_r).
\end{eqnarray}
Clearly, we must assume that $$\alpha_0<1$$ since for $\alpha_0\ge 1$ the complex $Y$ is either empty or has one vertex, a.a.s. (see Example \ref{emptyset}). 

\subsection{The density invariants} 
Let $S$ be a fixed finite simplicial $r$-dimensional complex. 
As usual, $f_i(S)$ denotes the number of $i$-dimensional faces in $S$. Define the following ratios (density invariants):
$$\mu_i(S) = \frac{f_0(S)}{f_i(S)}, \quad i=0, 1, \dots, r.$$
We do not exclude the case when $f_j(S)=0$; then $\mu_j(S)=\infty$. The 0-th number is always one, $\mu_0(S)=1$. 
Compare \cite{CCFK}, Definition 11.

\begin{lemma}\label{notembed}
If \begin{eqnarray}\label{less}
\sum_{i=0}^r \, \frac{\alpha_i}{\mu_i(S)} >1,\end{eqnarray}
then the probability that $S$ is embeddable into $Y\in Y_r(n, \p)$ tends to zero as $n\to \infty$. 
\end{lemma}
\begin{proof}
Let $F_0(S)$ denote the set of vertices of $S$. An embedding of $S$ into $\Delta_n^{(r)}$ is determined by an embedding $J: F_0(S) \to [n]$. 
For any such embedding define a random variable $X_J: Y_r(n,\p) \to \R$ given by 
$$X_J(Y) = \left\{
\begin{array}{ll}
1, & \mbox{if} \quad Y\supset J(S), \\
0, & \mbox{otherwise}.
\end{array}
\right.
$$
Then $X=\sum_J X_J$ is the random variable counting the number of isomorphic copies of $S$ in $Y$. One has 
$$\E(X_J) = \prod_{i=0}^r p_i^{f_f(S)}$$
by Lemma \ref{cont}. Thus we have 
\begin{eqnarray*}
\E(X) &=& \binom n {f_0(S)} \cdot f_0(S)!\cdot \prod_{i=0}^r p_i^{f_f(S)}  \\  \\ &\sim& n^{f_0(S)-\sum_{i=0}^r \alpha_if_i(S)} \\  \\ &=&
\left( n^{1-\sum_{i=0}^r\frac{\alpha_i}{\mu_i(S)}}\right)^{f_0(S)}.
\end{eqnarray*}
We see that (\ref{less1}) implies $\E(X)\to 0$. Hence $$\PP(X>0)\le \E(X)$$ also tends to zero. 
\end{proof}

\subsection{The density domains}
Consider the Euclidean Space $\R^{r+1}$ with coordinates $(\alpha_0, \alpha_1, \dots, \alpha_r)$. 

\begin{definition}
For a finite simplicial complex $S$ of dimension $\le r$, we denote by $$\Mu(S)\subset \R^{r+1}$$ the convex domain given by the following inequalities:
\begin{eqnarray*}
\alpha_0+ \frac{\alpha_1}{\mu_1(S)} + \frac{\alpha_2}{\mu_2(S)}  + \dots + \frac{\alpha_r}{\mu_r(S)} <1, \\ \\
\alpha_0\ge 0, \quad \alpha_1\ge 0, \quad \dots, \quad \alpha_r\ge 0. 
\end{eqnarray*}
\end{definition}

The domain $\Mu(S)$ is a simplex of dimension $r+1$ which has the origin as one of its vertices and the other vertices are of the form $\mu_i(S)e_i$ where $e_i = (0, \dots, 0, 1, 0, \dots, 0)$ with 
$1$ standing on the $i$-th position, where $i=0, 1, \dots, r$. 

We may restate Lemma \ref{notembed} as follows:

\begin{lemma}\label{cont1}
If the vector of exponents ${\aalpha} =(\alpha_0, \dots, \alpha_r)$ does not belong to the closure 
$${\aalpha}\, \notin \, \overline{\Mu(S)}$$ then the complex 
$S$ is not embeddable a.a.s. into a random simplicial complex $Y\in Y_r(n, \p)$, where $\p=n^{-\aalpha}$, i.e. 
$\p=(n^{-\alpha_0}, n^{-\alpha_1}, \dots, n^{-\alpha_r}).$
\end{lemma}


Next we define the domain $\tilde \Mu(S)\subset \R^{r+1}$ as the intersection
\begin{eqnarray}
\tilde \Mu(S) = \bigcap_{T\subset S} \Mu(T).
\end{eqnarray}
Here $T$ runs over all subcomplexes of $S$. Clearly, $\tilde \mu(S)$ is an $(r+1)$-dimensional convex polytope. 

\begin{lemma}\label{cont2}
If $$\aalpha = (\alpha_0, \alpha_1, \dots, \alpha_r)\in \tilde \Mu(S)$$ then $S$ is embeddable into a random complex $Y\in Y_r(n, \p)$ where $\p=n^{-\aalpha}$, a.a.s.
\end{lemma}

\begin{proof} As in the proof of Lemma \ref{cont}, 
let $F_0(S)$ denote the set of vertices of $S$ and let $J: F_0(S) \to [n]$ be an embedding.  Any such embedding uniquely determines a simplicial embedding $J: S \to \Delta_n$. 
Consider a random variable $X_J: Y_r(n,\p) \to \R$ given by 
$$X_J(Y) = \left\{
\begin{array}{ll}
1, & \mbox{if} \quad Y\supset J(S), \\ \\
0, & \mbox{otherwise}.
\end{array}
\right.
$$
Then $X=\sum_J X_J$ counts the number of isomorphic copies of $S$ in $Y$ and  
$\E(X_J) = \prod_{i=0}^r p_i^{f_f(S)}$
by Lemma \ref{cont}. Hence, $$\E(X) \sim n^{f_0(S)}\cdot \prod_{i\ge 0} p_i^{f_i(S)}.$$

We shall use the Chebyshev inequality 
\begin{eqnarray}\label{cheb}\PP(X=0) \le \frac{\rm Var(X)}{\E(X)^2}.\end{eqnarray}
One has 
\begin{eqnarray}\label{two}
{\rm Var}(X) = \E(X^2) - \E(X)^2 =
\sum_{J,J'}\left[ \E(X_JX_J') - \E(X_J)\E(X_{J'})\right].\end{eqnarray}
Given two simplicial embeddings $J,J': S \to \Delta_n$, the product $X_JX_{J'}$ is a 0-1 random variable, it has value 1 on a subcomplex $Y\subset \Delta_n^{(2)}$ if and only if $Y$ contains the union $J(S)\cup J'(S)$. Thus, by Lemma \ref{cont}, we have 
$$\E(X_JX_{J'})= \prod_{i\ge 0} p_i^{2f_i(S)-f_i(T')},$$
where $T'$ denotes the intersection $J(S)\cap J'(S)$. 
If $J(S)$ and $J'(S)$ are disjoint then $\E(X_JX_{J'})$ equals $\E(X_J)\cdot \E(X_{J'})$; thus in formula (\ref{two}) we may assume that $J$ and $J'$ are such that the intersection $J(S)\cap J'(S)\not= \emptyset$

Denote by $T$ by the subcomplex $T=J^{-1}(T')\subset S$.  
For a fixed subcomplex $T\subset S$ the number of pairs of embeddings $J,J': S\to \Delta_n$ such that $J^{-1}(J(S)\cap J'(S))=T$ is bounded above by 
$$C_Tn^{2f_0(S)-f_0(T)}$$
where $C_T$ is the number of isomorphic copies of $T$ in $S$. 

Thus we obtain
$${\rm Var}(X)  \le \sum_{T\subset S} C_T n^{2f_0(S) -f_0(T)} \prod_{i\ge 0} p_i^{2f_i(S)-f_i(T)}.$$
On the other hand,
$$\E(X) \ge \frac{1}{2} n^{f_0(S)}\prod_{i\ge 0} p_i^{f_i(S)}.$$
Therefore,
\begin{eqnarray*}
\frac{{\rm Var(X)}}{\E(X)^2} 
&\le& 
4\cdot \sum_{T\subset S} C_T n^{-f_0(T)}\cdot \prod_{i\ge 0} p_i^{-f_i(T)} \\
&\le&
4\cdot \sum_{T\subset S} C_T n^{-f_0(T) + \sum_{i\ge 0} \alpha_i f_i(T)}\\
&= &
4\cdot \sum_{T\subset S} C_T n^{\left[\sum_{i\ge 0} \frac{\alpha_i}{\mu_i(T)} -1\right] f_0(T)}\\ 
\end{eqnarray*}
Here $T$ runs over all nonempty subcomplexes of $S$. 
If $(\alpha_0, \dots, \alpha_r)\in \tilde \Mu(S)$ then for any $T\subset S$ we have 
$$\sum_{i\ge 0} \frac{\alpha_i}{\mu_i(T)} <1.$$
Thus we see that the ratio $\frac{{\rm Var}(X)}{\E(X)^2}$ tends to zero as $n\to \infty$. The result now follows from (\ref{cheb}). 
\end{proof}

We may summarise the obtained results as follows:

\begin{lemma}\label{summary} Let $S$ be a fixed finite simplicial complex of dimension $\dim S \le r$. 

\begin{enumerate}
\item If $\alpha\in \tilde \Mu(S)$ then a random simplicial complex $Y\in Y_r(n, n^{-\alpha})$ contains $S$ as a simplicial subcomplex, a.a.s.

\item If $\alpha\notin {\it {Closure}}(\tilde \Mu(S))$ then a random simplicial complex $Y\in Y_r(n, n^{-\alpha})$ does not contain $S$ as a simplicial subcomplex, a.a.s.
\end{enumerate}
\end{lemma}

The first statement repeats Lemma \ref{cont2}. The second statement follows from Lemma \ref{cont1} and from the equality
$${\it {Closure}}(\tilde \Mu(S)) = \bigcap_{T\subset S} {\it {Closure}}(\Mu(T)).$$

\subsection{The reduced density domain}

Since $\mu_0(S)=1$ the simplex $\Mu(S)\subset \R^{r+1}$ contains the point $(1, 0, \dots, 0)$ as one of its vertices and 
$\Mu(S)$ is the cone with apex $(1, 0, \dots, 0)\in \R^{r+1}$ over the simplex 
\begin{eqnarray}
\mu(S) =\{(\alpha_1, \dots, \alpha_r); \, \sum_{i+1}^r \, \frac{\alpha_i}{\mu_i(S)} <1, \quad \alpha_i \ge 0\}\, \subset \R^r.
\end{eqnarray}
Hence the convex domain $$\tilde \Mu(S)\subset \R^r$$ is a cone with apex 
$(1, 0, \dots, 0)\in \R^{r+1}$ over the domain $\tilde \mu(S)\subset \R^r$ which is defined as the intersection
\begin{eqnarray}\label{wmu}
\tilde \mu(S) = \bigcap_{T\subset S} \mu(T).  
\end{eqnarray}
We call $\tilde \mu(S)\subset \R^r$ {\it the reduced density domain} associated to $S$. 

For a subset of vertices $W\subset V=V(S)$, denote by $S_W\subset S$ the simplicial complex induced on $W$. Then 
\begin{eqnarray}\label{induced}
\tilde \mu(S) = \bigcap_{W\subset V} \mu(S_W).
\end{eqnarray}
This follows from the observation that for a simplicial subcomplex $T\subset S$ one has 
$$\mu_i(T) \ge \mu_i(S_W)$$
and therefore 
$$\mu(T) \supset \mu(S_W),$$
where $W=V(T)\subset V(S)$ is the set of vertices of $T$. 

\subsection{The Invariance Principle}

By Lemmas \ref{cont1} and \ref{cont2}, the extended density domain $\tilde \Mu(S)$ controls embedability of $S$ into a random simplicial complex $Y\in Y_r(n,n^{-\alpha})$.  and since th is a cone with vertex $(1, 0, \dots, 0)$ we obtain the following corollary:

\begin{corollary}
If $\alpha = (\alpha_0, \alpha_1, \dots, \alpha_r)\in \R^{r+1}$ and $ \alpha' = (\alpha'_0, \alpha'_1, \dots, \alpha'_r)\in \R^{r+1} $ lie on a line passing trough $(1, 0, \dots,0)\in \R^{r+1}$ then the probability spaces 
$Y_r(n, n^{-\alpha})$ and $Y_r(n, n^{-\alpha'})$ have identical embedability properties with respect to any fixed finite simplicial complex $S$, $\dim S \le r$, i.e. $S$ embeds into 
$Y\in Y_r(n, n^{-\alpha})$ a.a.s. if and only if $S$ is embeds into 
$Y\in Y_r(n, n^{-\alpha'})$, a.a.s.
\end{corollary}

Thus, instead of a vector of exponents $\alpha=(\alpha_0, \alpha_1, \dots, \alpha_r)$ we may consider the vector 
$$\alpha'=(0, \frac{\alpha_1}{1-\alpha_0}, \frac{\alpha_2}{1-\alpha_0}, \dots, \frac{\alpha_r}{1-\alpha_0})$$
which has the first coordinate $0$, i.e. in this case $p_0=1$. 

{\bf Conjecture:} 
{\it We conjecture that all geometric and topological properties of the the random complex $Y_r(n, n^{-\alpha})$ remain invariant when the multi-exponent $\alpha$ moves along any line passing through the point $(1, 0, \dots, 0)$. }

\subsection{Examples}

\begin{example}{\rm 
Let $S$ be a closed triangulated surface, $r=2$. Then the number of edges $e$ and the number of faces $f$ are related by $3f=2e$. We obtain that
$$3\mu_1(S) = 2\mu_2(S)$$
i.e. the simplex $\mu(S)\subset \R^2$ has a fixed slope independent of the topology of the surface and of the details of a particular triangulation. 

More specifically, if $S$ is a closed surface then 
\begin{eqnarray}
\mu_1(S) = \frac{1}{3} + \frac{\chi(S)}{e(S)}, \quad \mu_2(S) = \frac{1}{2} + \frac{\chi(S)}{f(S)},
\end{eqnarray}
where $\chi(S)$ is the Euler characteristic of $S$ and $e(S)$, $f(S)$ denote the numbers of edges and faces (i.e. 2-simplexes) in $S$. This follows from 
\cite{CCFK} page 132 and \cite{CFH}, \S 2.
}
\end{example}
The following Figures \ref{positive}, \ref{zero}, \ref{negative} show the density domains for closed surfaces $\Sigma$ depending on whether $\chi(\Sigma)$ is positive, zero or negative. 
\begin{figure}[h]
\centering
\includegraphics[width=0.4\textwidth]{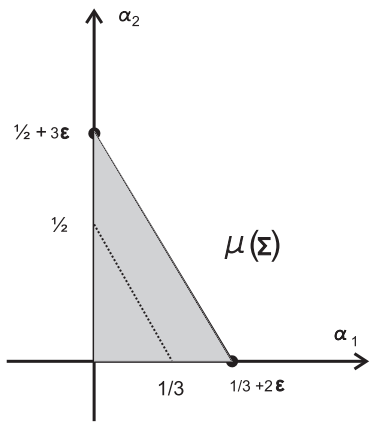}
\caption{The reduced density domain $\mu(\Sigma)$ for a triangulation of a closed surface with $\chi(\Sigma)> 0$. The number $\epsilon>0$ depends on the genus and on the number of 2-simplexes in the triangulation} \label{positive}
\end{figure}
\begin{figure}[h]
\centering
\includegraphics[width=0.5\textwidth]{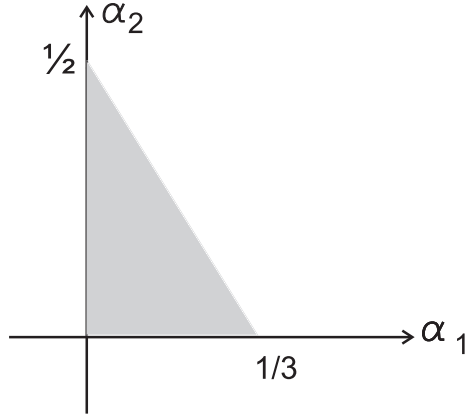}
\caption{The reduced density domain of a triangulated closed surface with $\chi(\Sigma)=0$ (i.e. $\Sigma$ is the torus or the Klein bottle).}\label{zero}
\end{figure}
\begin{figure}[h]
\centering
\includegraphics[width=0.5\textwidth]{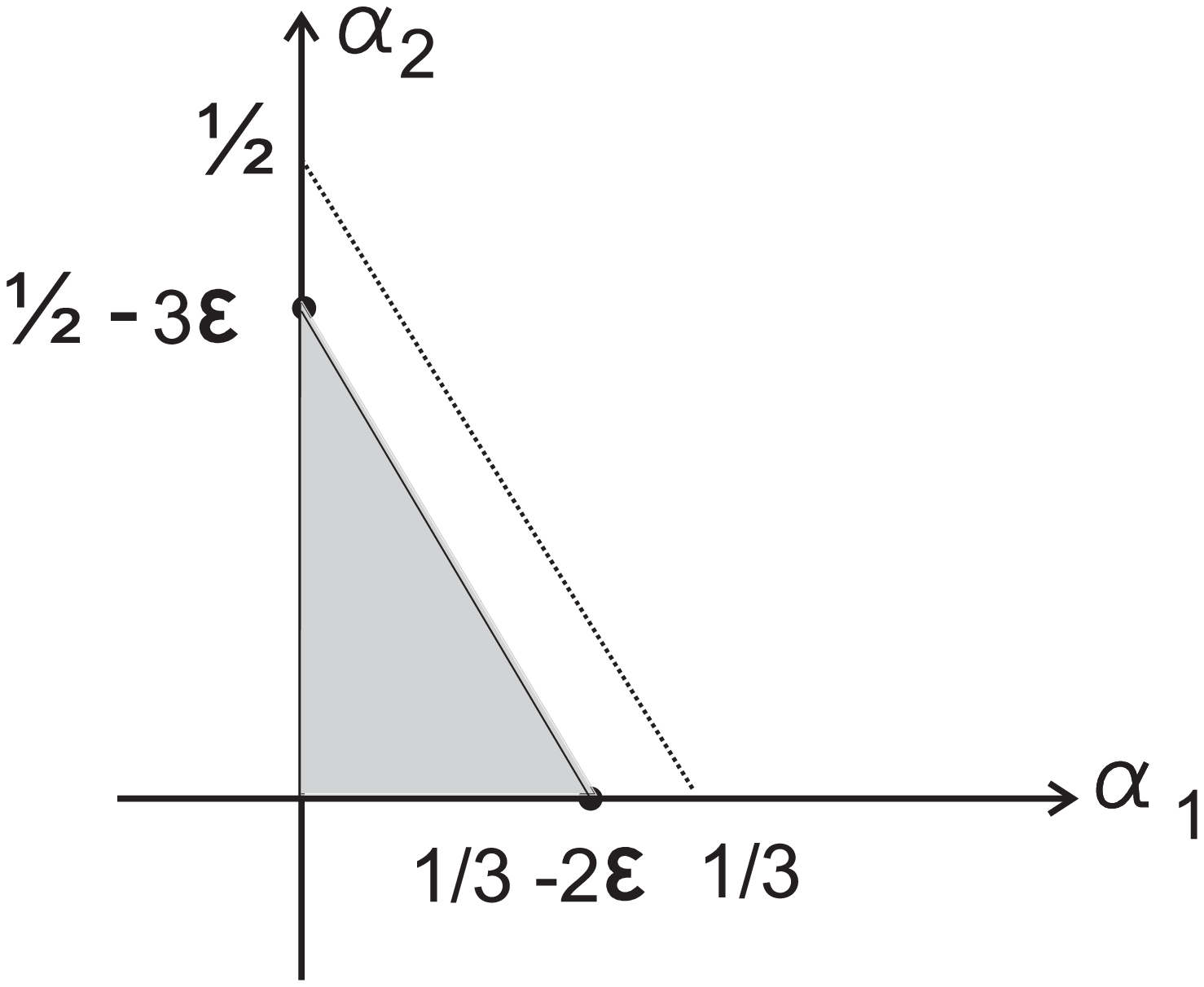}
\caption{The reduced density domain $\mu(\Sigma)$ of a triangulated closed surface with $\chi(\Sigma)<0$. 
The number $\epsilon>0$ depends on the genus and on the number of 2-simplexes in the triangulation.}\label{negative}
\end{figure}

\begin{example}{\rm 
Let $Z=P^2\cup D^2$ be the following 2-complex. Here $P^2$ is a triangulated real projective plane having a cycle $C$ of length 5 representing the non-contractible loop. 
$D^2$ is a triangulated disc with boundary of length 5 which is identified with $C$.
To compute $\mu_i(Z)$ we shall use the formulae
\begin{eqnarray}\label{ll}
\mu_1(Z) = \frac{1}{3} + \frac{\chi(Z)+ L(Z)/3}{e(Z)}, \quad \mu_2(Z) = \frac{1}{2} + \frac{\chi(Z)+L(Z)/2}{f(Z)},
\end{eqnarray}
where $e(Z)$ and $f(Z)$ denote the numbers of edges and faces of $Z$ and $L(Z)$ denotes $$\sum_e [2-\deg_Z(e)].$$ 
Here $e$ runs over the edges of $Z$ and $\deg_Z(e)$ is the number of faces of $Z$ containing $e$, see formula (2) in \cite{CF1} and formula (8) in \cite{CFH}. 
In our case, $L(Z) =-5$ and $\chi(Z) = 2$; therefore, 
\begin{eqnarray}
\mu_1(Z) = \frac{1}{3} + \frac{1}{3e(Z)}, \quad \mu_2(Z) = \frac{1}{2} - \frac{1}{2f(Z)}. 
\end{eqnarray}
\begin{figure}[h]
\centering
\includegraphics[width=0.4\textwidth]{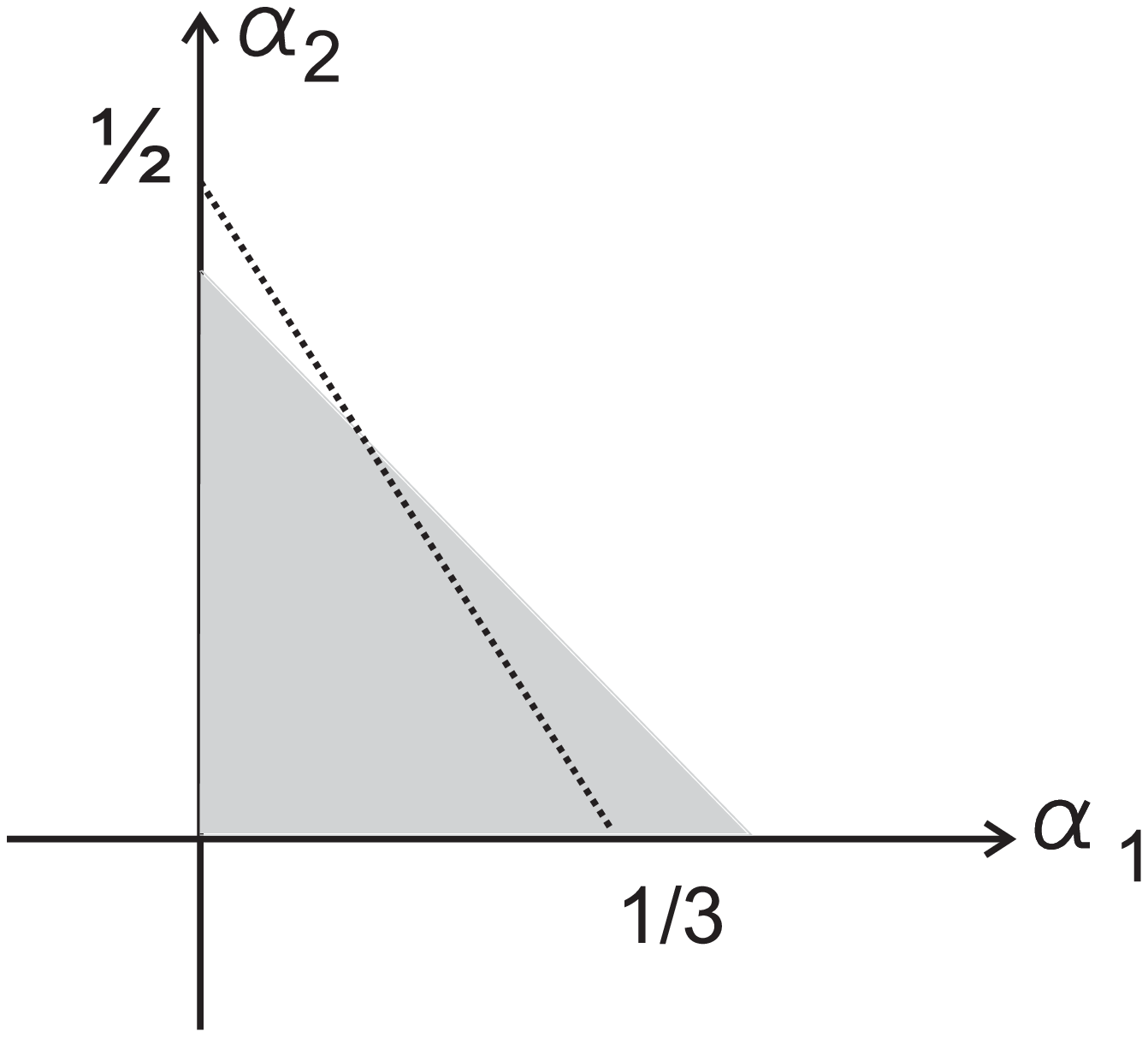}
\caption{The reduced density domain of the complex $Z$ has a lower slope than in the case of surfaces.}\label{zed}
\end{figure}
}
\end{example}

\begin{example} \label{exnice}
\begin{figure}[h]
\centering
\includegraphics[width=0.7\textwidth]{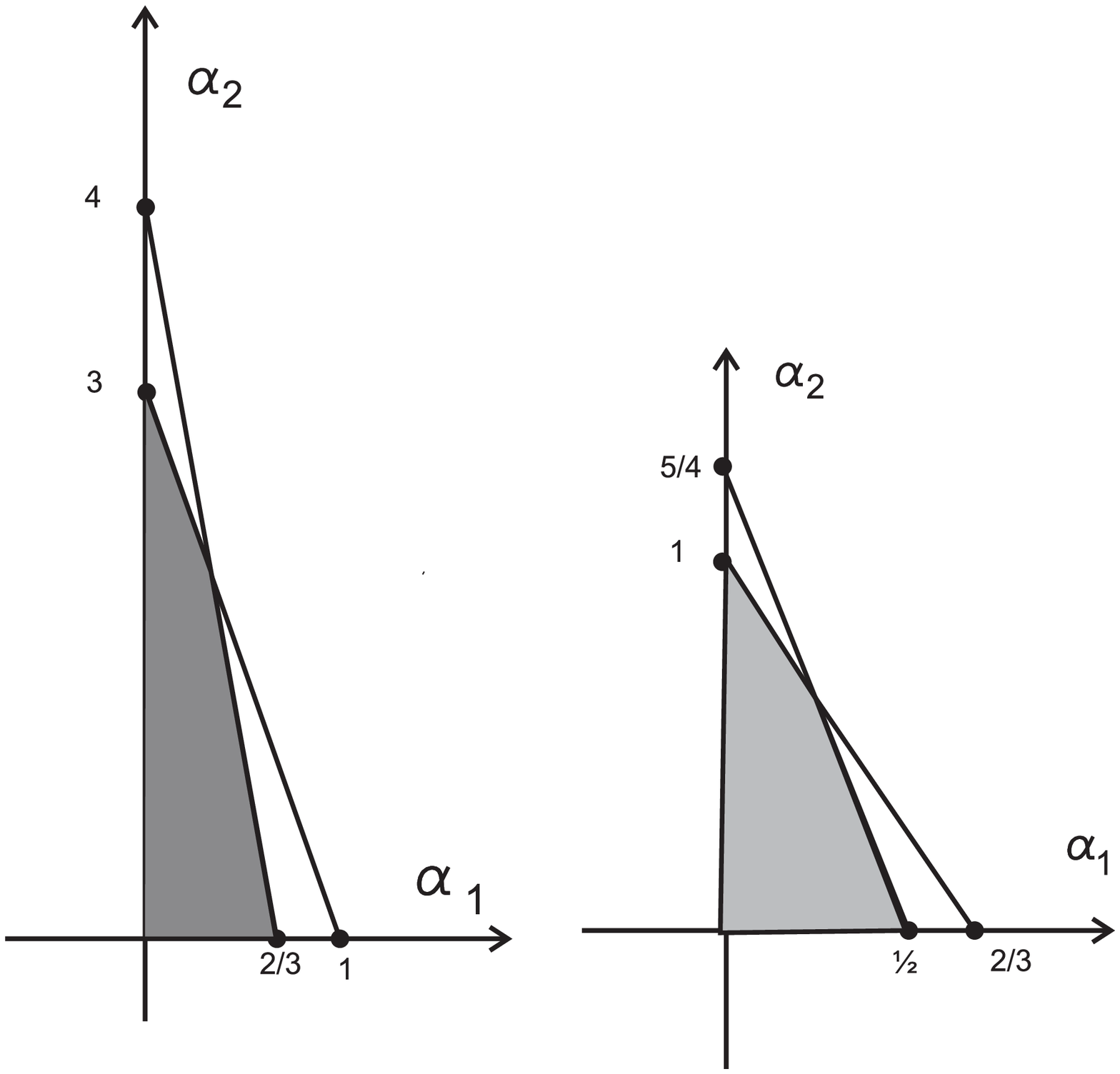}
\caption{The reduced density domain $\tilde \mu(S_t)$ of the complex $S_t$ of Example \ref{exnice} with $t=3$ (left) and $t=4$ (right).}\label{nice}
\end{figure}
{\rm For an integer $t$, let $S_t$ be a 2-complex constructed as follows. 
The vertex set $F_0(S_k)$ is $\{1, 2, \dots, t+1\}$, the set of 1-simplexes $F_1(S_t)$ is the set of all pairs $(i,j)$ 
where $1\le i<j\le t+1$ 
(i.e. the 1-skeleton of $S_t$ is a complete graph on $t+1$ vertices), and the set of 2-simplexes $F_2(S_t)$ consists of triples $(i,j,k)$ where $1\le i<j<k\le t$. 
To describe the reduced density domain $\tilde \mu(S_t)$ we shall use the formula (\ref{induced}). 
Consider a subset $W\subset F_0(S_t)=\{1, \dots, t+1\}.$ If $W$ does not contain the vertex $t+1$ 
then the induced complex $S_W$ is the 2-skeleton of the $\tau$-dimensional simplex where $|W|=\tau+1$ and we have 
\begin{eqnarray}
\mu_1(S_W)  \, = \, \frac{\tau+1}{{\binom {\tau+1} 2}}=\frac{2}{\tau}, 
\quad \mu_2(S_W)  \, = \, \frac{\tau+1}{{\binom {\tau+1} 3}}
\end{eqnarray}
If $W$ contains the last vertex $t+1$ then 
\begin{eqnarray}
\mu_1(S_W)  \, = \, \frac{\tau+1}{{\binom {\tau+1} 2}}=\frac{2}{\tau}, \quad \mu_2(S_W)  \, = \, \frac{\tau+1}{{\binom {\tau} 3}}.
\end{eqnarray}
In the first case, $\tau\le t-1$ and in the second case $\tau \le t$.
We see that 
$$\min_W \mu_1(S_W)  = \frac{2}{t}$$
(is achieved for $W=W_1=\{1, \dots, t+1\}$) and 
$$\min_W \mu_2(S_W) = \frac{t}{\binom t 3}$$
(is achieved for $W=W_2= \{1, \dots, t\}$). 
The two lines given by the equations
$$\frac{\alpha_1}{\mu_1(S_{W_i})} + \frac{\alpha_2}{\mu_2(S_{W_i})}=1,\quad i=1, 2,$$
intersect at the point
$$(\alpha_1, \alpha_2) = \left(\frac{1}{t}, \frac{3(t+1)}{t(t-1)(t-2)}\right).$$
It is easy to check that this point $(\alpha_1, \alpha_2) $ satisfies the inequality 
$$\frac{\alpha_1}{\mu_1(S_W)} + \frac{\alpha_2}{\mu_2(S_W)}\le 1,$$
for arbitrary subset $W\subset \{1, \dots, t+1\}$. This argument shows that 
in this example 
$$\tilde \mu(S_t) = \mu(S_{W_1})\cap \mu(S_{W_2})$$ and 
justifies our picture Figure \ref{nice}. 
}

\end{example}

%

\section{Balanced simplicial complexes}

\begin{definition}
We shall say that an $r$-dimensional simplicial complex $S$ is balanced if 
$$\tilde \mu(S) = \mu(S)\subset \R^r.$$
\end{definition}

The complex $Z$ of Example \ref{exnice} is unbalanced because its density domain $\tilde \mu(Z)$ is not a simplex, see Figure \ref{nice}. 

\begin{lemma} The following properties are equivalent: 

a) $S$ is balanced;

b) for any subcomplex $T\subset S$ one has 
$\mu(S)\subset \mu(T)$. 

c) for any subcomplex $T\subset S$ and for any $i=1, \dots, r$ one has 
$\mu_i(T) \ge \mu_i(S)$. 
\end{lemma}

The proof is obvious.

\begin{example} {\rm 
Let $S$ be a 2-dimensional triangulated disc having $v$ vertices such that among them $v_i$ are internal. It is easy to check (using the Euler - Poincar\'e formula) that 
$e=2v +v_i -3$ and $f = v+v_i-2$. Hence, 
\begin{eqnarray}\label{disc}
\mu_1(S) = \frac{v}{2v+v_i - 3}, \quad \mu_2(S) = \frac{v}{v+v_i - 2}.\end{eqnarray}
Let us assume that $v_i >3$; then $\mu_1(S) <1/2$ and $\mu_2(S) <1$. 
Suppose that there exists a proper subdisc $T\subset S$ containing all the internal vertices. 
Then 
$$\mu_1(T) = \frac{v'}{2v'+v_i - 3}, \quad \mu_2(S) = \frac{v'}{v'+v_i - 2}$$
where $v'<v$ and we see that 
$$\mu_i(T)< \mu_i(S), \quad \mbox{for}\quad i=1, 2.$$

This argument shows that there exist unbalanced triangulations of the disc. 
}
\end{example}

\begin{theorem}\label{surfacebalanced}
Any triangulation of a closed surface $S$ with $\chi(S)\ge 0$ is balanced. 
\end{theorem}

\begin{proof} Let $S$ be a triangulated closed surface, $\chi(S) \ge 0$, and let $T\subset S$ be a proper subcomplex. 
We want to show that 
\begin{eqnarray}\label{less1}
\mu_i(T)\ge \mu_i(S), \quad \mbox{for}\quad i=1, 2.\end{eqnarray} 
Using formulae (\ref{ll}) and our assumption $\chi(S)\ge 0$ 
we see that (\ref{less1}) would follow from the inequalities 
\begin{eqnarray}\label{2ineq} L(T) \ge 3\chi(S,T), \quad L(T)\ge 2\chi(S,T),\end{eqnarray}
since $e(T)\le e(S)$ and $f(T)\le f(S)$. Here $\chi(S,T)=\chi(S) - \chi(T)$. 
Clearly, every edge of $T$ has degree 0,1, or 2 and hence $L(T) \ge 0$; therefore, (\ref{2ineq}) follows automatically if $\chi(S,T)\le 0$. 
In the case $\chi(S,T) >0$ it is enough to show the left inequality in (\ref{2ineq}).

Let $N\supset T$ be a small tubular neighbourhood of $T$ in $S$. We shall denote by $K= S- \int(N)$ the closure of the complement of $N$ in $S$ and apply Proposition 3.46 from \cite{Hatcher}. 
This Proposition states that $H^i(K; \Z_2)$ is isomorphic to $H_{2-i}(S,S-K;\Z_2)$ and is valid without the assumption that $S$ is orientable since we take $\Z_2$ coefficients. 
Thus, 
$$\chi(K) = \chi(S, S-K) = \chi(S, T).$$
%
%
%
%
If we denote by $k=b_0(K) = b_0(S-T)$ the number of connected components of $S-T$ then 
$$\chi(S,T) = \chi(T) \le b_0(K) = k$$ 
and 
(\ref{less1})  would follow from $$L(T) \ge 3k.$$
Note that $$L(T)=e_1(T)+2e_0(T)$$ where $e_i(T)$ denotes the number of edges of $T$ which have degree $i=0,1$. 

Consider the $j$-th connected component of the complement $S-T$; its set-theoretic boundary is a graph which is the image of a simplicial map $C_j\to S$ where $C_j$ is a triangulation of the circle. Denoting by $|C_j|$ the number of edges we see that 
$$\sum_{j=1}^k |C_j| = e_1(T) +2e_0(T)=L(T).$$
Indeed, the image of the maps $\sqcup_j C_j\to S$ is the union of edges of $T$ having degree 0 and 1 and each edge of degree 2 is covered twice. 
Clearly, $|C_j|\ge 3$ for each $C_j$ and 
the inequality $L(T) \ge 3k$ follows. 
\end{proof}

\begin{remark} {\rm It is easy to show that the assumption $\chi(S) \ge 0$ of Theorem \ref{surfacebalanced} is necessary. 
More specifically, {\it any closed surface with $\chi(S)<0$ admits a non-balanced triangulation.} For the proof, see \cite{CFH}, \S 2. }
\end{remark}


\begin{definition}
We shall say that $S$ is strictly balanced if for any proper subcomplex $T\subset S$ one has 
$$\mu_i(T)\ge \mu_i(S)\quad\mbox{for}\quad i=1, 2, \dots, r$$
and at least one of these inequalities is strict. 
\end{definition}

Let $S$ be a simplicial complex of dimension $r$. {\it The degree} of an $i$-dimensional simplex $\sigma\subset S$ is the number of $(i+1)$-dimensional simplexes containing $\sigma$; we denote this degree $\deg \sigma$. 

\begin{lemma}\label{reg}
Let $S$ be a connected $r$-dimensional simplicial complex with the property that the degree of any $i$-dimensional simplex $\sigma\subset S$ depends only on $i$. 
Then $S$ is strictly balanced. 
\end{lemma}

For the proof we need an expression of the density invariants $\mu_i(S)$ in terms of the average degree of simplexes which is described in the following lemma.

For a simplicial complex $S$ we denote by $\bar d_i(S)$ the ratio
$$
\frac{\sum\limits_{\dim \sigma=i} \deg \sigma}{f_i(S)}. 
$$
It has the meaning of the average degree of $i$-dimensional simplexes in $S$. 

\begin{lemma} For an $r$-dimensional simplicial complex $S$ one has 
\begin{eqnarray}\label{average}
\mu_i(S) = \frac{(i+1)!}{\bar d_0(S) \cdot \bar d_1(S) \cdots \bar d_{i-1}(S)}\quad \mbox{where} \quad i=1, 2, \dots, r.
\end{eqnarray}
\end{lemma}

\begin{proof} We observe that
$$\sum_{\dim \sigma=i} \deg (\sigma)= (i+2) \cdot f_{i+1}(S)$$
and therefore 
$$\bar d_i(S) = (i+2)\cdot \frac{f_{i+1}(S)}{f_i(S)}.$$
Multiplying these equalities we obtain 
$$\bar d_0(S) \cdot \bar d_1(S) \cdots \bar d_{i-1}(S) = (i+1)! \cdot \frac{f_i(S)}{f_0(S)}.$$
The last equality is equivalent to the claim of the lemma.

\end{proof}

\begin{proof}[Proof of Lemma \ref{reg}] 
By assumption, $\deg^S_i (\sigma) = d_i$ for any simplex $\sigma\subset S$, $\dim \sigma=i$. Let $T\subset S$ be a proper subcomplex. 
Then $\deg_i^T(\sigma) \le d_i$ for any simplex $\sigma\subset T$, $\dim \sigma=i$. Moreover, $\deg^T_i(\sigma) < d_i$ for some $\sigma$ and some $i$ 
(for example, for the minimal $i$ such that 
$f_{i+1}(T)<f_{i+1}(S)$). 
This shows that 
$$d_i \ge \bar d_i(T), \quad i=0, 1, \dots, r-1.$$
and at least one of these inequalities is strict. Using the formula (\ref{average}) we obtain that 
$$\mu_i(S) \le \mu_i(T), \quad i= 1, \dots, r$$
and at least of these inequalities is strict. 
\end{proof}

\section{Dimension of a random simplicial complex}

Let $S$ be the abstract simplex of dimension $s\le r$. Then $S$ is strictly balanced (by Lemma \ref{reg}). 
The embeddability of $S$ into a random complex $Y\in Y_r(n, \p)$ means that $\dim Y \ge s$, hence we may use Lemma \ref{summary} to answer the question about the dimension of $Y$. 
We have 
$$\mu_i(S) = \frac{s+1}{\binom {s+1}{i+1}}, \quad i=0, 1, \dots, s$$
and 
$$\mu_i(S) =\infty\quad \mbox{for}\quad i=s+1, \dots, r.$$
Therefore, applying Lemmas \ref{reg} and \ref{summary} we see that {\it the dimension of a random simplicial complex $Y\in Y_r(n, n^{-\alpha})$ satisfies $\dim Y\ge s$ if }
\begin{eqnarray}\label{dimension}\sum\limits_{i=0}^s \alpha_i \cdot \frac{\binom {s+1} {i+1} }{s+1} \, < \, 1.\end{eqnarray}
Here $\alpha=(\alpha_0, \alpha_1, \dots, \alpha_r)$, $p_i=n^{-\alpha_i}$. 
The inequality (\ref{dimension})  can be rewritten as 
$$\alpha_0 + \alpha_1 \cdot \frac{s}{2} + \alpha_2\cdot \frac{s(s-1)}{6} +\dots +\alpha_{s}\cdot \frac{1}{s+1} <1.$$

For a vector $\alpha=(\alpha_0, \alpha_1, \dots, \alpha_r)$ with $\alpha_i\ge 0$ define the quantities 
\begin{eqnarray}
D_s(\alpha) \, =\,  \sum_{i=0}^r \, \frac{\binom {s+1} {i+1} } {s+1} \cdot \alpha_i,\quad\quad s=0, 1, \dots, r.
\end{eqnarray}
Note that $\binom {s+1} {i+1} =0$ for $i>s$. It is easy to check that 
$$\frac{\binom {s+1} {i+1} } {s+1} \le \frac{\binom {s+2} {i+1} } {s+2}$$
and hence one has
\begin{eqnarray}\label{increase} 
D_0(s) \le D_1(s) \le D_2(s) \le \dots \le D_r(s).
\end{eqnarray}
Here 
\begin{eqnarray*}
&D_0(s)& = \alpha_0,\\
&D_1(s)& = \alpha_0 +\frac{1}{2}\cdot\alpha_1,\\
&D_2(s)& = \alpha_0 +\alpha_1 + \frac{1}{3}\cdot \alpha_2,\\
&D_3(s)& = \alpha_0 +\frac{3}{2} \cdot \alpha_1 + \alpha_2 + \frac{1}{4}\cdot \alpha_3.\\
\end{eqnarray*}

We obtain the following result: 

\begin{figure}[h]
\centering
\includegraphics[width=0.65\textwidth]{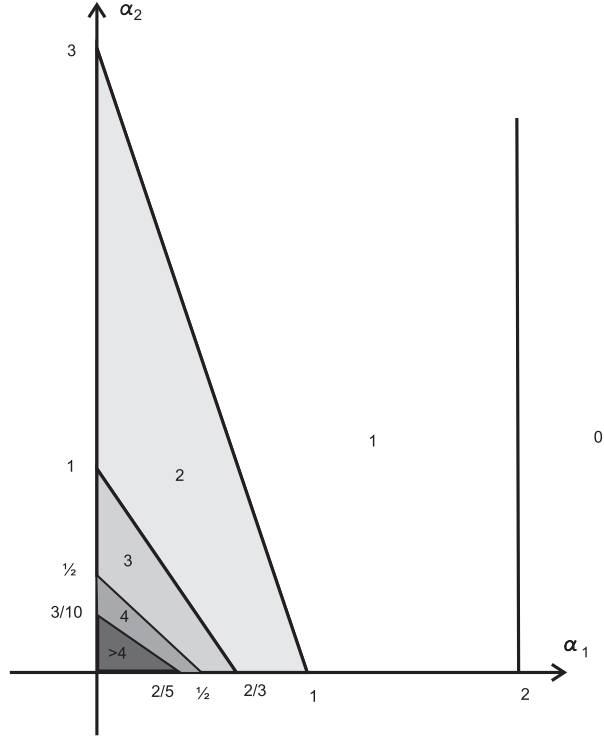}
\caption{Dimension of the random simplicial complex for various values of parameters $\alpha_1, \alpha_2$.}\label{dimension1}
\end{figure}

\begin{corollary}\label{dim1}
Given a multi-index $\alpha =(\alpha_0, \alpha_1, \dots, \alpha_r)$, $\alpha_i\ge 0$ and an integer $s$ satisfying $0\le s\le r$.  

(1) If $D_s(\alpha)<1$ then dimension of a random complex $Y\in Y_r(n, n^{-\alpha})$ satisfies
$\dim Y\, \ge \, s, $ 
a.a.s.

(2) If $D_s(\alpha) >1$ then  dimension of a random complex $Y\in Y_r(n, n^{-\alpha})$ satisfies
$\dim Y\, < \, s. $

(3) The convex domain given by the inequalities
$$D_s(\alpha) <1<D_{s+1}(\alpha), \quad \mbox{and}\quad \alpha_0\ge 0, \, \alpha_1\ge 0, \, \dots, \, \alpha_r\ge 0$$ describes the area of the multi-parameter $\alpha\in \R^{r+1}$ such that dimension of a random complex 
$Y\in Y_r(n, n^{-\alpha})$ satisfies $\dim Y= s$, a.a.s.

\end{corollary}

In particular we see that $\dim Y\ge 0$ if $\alpha_0<1$ in accordance with the result of Example \ref{emptyset}. 

As an illustration of Corollary \ref{dim1} consider the special case when $\alpha_0=0$ and $$\alpha_3=\alpha_4=\dots =\alpha_r=0,$$ i.e. we have only two nonzero parameters $\alpha_1$ and $\alpha_2$. 
Then Corollary \ref{dim1} implies:

\begin{itemize}
\item $\dim Y = 1$ if $\alpha_1<2$ and $\alpha_1+\frac{1}{3}\alpha_2 > 1$, a.a.s.;

\item $\dim Y = 2$ if $\alpha_1+\frac{1}{3}\alpha_2 < 1$ and $\frac{3}{2}\cdot \alpha_1 +  \alpha_2 >1$, 
a.a.s.;

\item $\dim Y = 3$ if $\frac{3}{2}\cdot \alpha_1 + \alpha_2 <1$ and $2\alpha_1+2\alpha_2>1$, a.a.s.; 

\item $\dim Y= 4$ if $2\alpha_1+2\alpha_2 <1$ and $\frac{5}{2}\cdot \alpha_1 +\frac{10}{3}\cdot \alpha_2 >1$, a.a.s.
 \end{itemize}
and so on.

Figure \ref{dimension1} depicts regions of the plane $(\alpha_1, \alpha_2)$ where dimension is $1, 2, 3, 4$ and $\ge 5$. Each of these regions is a polygonal convex domain with vertices in rational points.

\bibliographystyle{amsalpha}

\vskip 1cm
A. Costa and M. Farber: 

School of Mathematical Sciences, Queen Mary University of London, London E1 4NS
\vskip 0.7cm

\end{document}